\documentclass[12pt]{article}
\usepackage{amssymb,amsthm,amsmath,latexsym}
\newtheorem{thm}{Theorem}

\newtheorem{lem}[thm]{Lemma}

\theoremstyle{remark}

\newcommand{\FF}{\mathbb{F}}

\newcommand{\cC}{\mathcal{C}}

\DeclareMathOperator{\image}{Im}
\DeclareMathOperator{\Aut}{Aut}
\DeclareMathOperator{\MAut}{MAut}
\DeclareMathOperator{\wt}{wt}

\DeclareMathOperator{\GL}{GL}

\begin{document}

\title{Classification of Self-Dual Codes of Length 36}

\author{
Masaaki Harada\thanks{
Department of Mathematical Sciences,
Yamagata University,
Yamagata 990--8560, Japan, and
PRESTO, Japan Science and Technology Agency (JST), Kawaguchi,
Saitama 332--0012, Japan. 
email: mharada@sci.kj.yamagata-u.ac.jp}
and 
Akihiro Munemasa\thanks{
Graduate School of Information Sciences,
Tohoku University,
Sendai 980--8579, Japan.
email: munemasa@math.is.tohoku.ac.jp}
}

\maketitle

\begin{abstract}
A complete classification of binary self-dual codes of length $36$ is given.
\end{abstract}

{\small
\noindent
{\it Key words and phrases:} self-dual code, weight enumerator, mass formula
\\
\noindent
{\it 2000 Mathematics Subject Classification:}
Primary: 94B05; Secondary: 94B75
}

\section{Introduction}

As described in~\cite{RS-Handbook},
self-dual codes are an important class of linear codes for both
theoretical and practical reasons.
It is a fundamental problem to classify self-dual codes
of modest lengths and 
much work has been done towards classifying self-dual codes over $\FF_q$
for $q=2$ and $3$,
where $\FF_q$ denotes the finite field of order $q$
and $q$ is a prime power
(see~\cite{RS-Handbook}).

Codes over $\FF_2$ are called {\em binary} and
all codes in this paper are binary unless otherwise noted.
The \textit{dual code} $C^{\perp}$ of a code 
$C$ of length $n$ is defined as
$
C^{\perp}=
\{x \in \FF_2^n \mid x \cdot y = 0 \text{ for all } y \in C\},
$
where $x \cdot y$ is the standard inner product.
A code $C$ is called \textit{self-orthogonal} 
if $C \subset C^{\perp}$,
and $C$ is  called \textit{self-dual} if $C = C^{\perp}$. 
A self-dual code $C$ is {\em doubly even} if all
codewords of $C$ have weight divisible by four, and {\em
singly even} if there is at least one codeword of weight $\equiv 2
\pmod 4$.
It is  known that a self-dual code of length $n$ exists 
if and only if  $n$ is even, and
a doubly even self-dual code of length $n$
exists if and only if $n$ is divisible by eight.
Two codes are {\em equivalent} if one can be
obtained from the other by permuting the coordinates.
An {\em automorphism} of $C$ is a permutation of the coordinates of $C$
which preserves $C$.
The set consisting of all automorphisms of $C$ is called the
{\em automorphism group} of $C$ and it is denoted by 
$\Aut(C)$. 

A classification of 
self-dual codes of lengths up to $30$ and
doubly even self-dual codes of length $32$
is known (see~\cite[Table I]{RS-Handbook}).
A classification of singly even
self-dual codes of length $32$ is given in~\cite{BR02}.
The classification is extended to length $34$~\cite{B06}.
Using the classification of self-dual codes of length
$34$ and minimum weight $6$,
extremal self-dual codes of length $36$, 
that is, those with minimum weight $8$, 
were classified in~\cite{MG08}.

The main aim of this paper is to give a complete classification
of self-dual codes of length $36$, confirming in particular,
the partial classification given in~\cite{MG08}.

\begin{thm}\label{thm:36}
There are $519492$ inequivalent self-dual codes of length $36$.
Of these $41$ are extremal, $58671$ have minimum weight $6$,
$436633$ have minimum weight $4$, and 
$24147$ have minimum weight $2$.
\end{thm}

Generator matrices of all inequivalent self-dual codes
of length $36$, as well as those of shorter lengths,
can be obtained electronically from~\cite{Data}.
As a summary, we list in Table~\ref{Tab:N}
the total number $\#_T$ of 
inequivalent self-dual codes of length $n$ and
the number $\#_d$ of 
inequivalent self-dual codes of length $n$ and minimum weight $d$
for $n=2,4,\ldots,36$.
All computer calculations in this paper
were done by {\sc Magma}~\cite{Magma}.

\begin{table}[th]
\caption{Numbers of self-dual codes}
\label{Tab:N}
\begin{center}
{\small
\begin{tabular}{c|c|cc||c|r|rrrr}
\noalign{\hrule height0.8pt}
$n$  
& \multicolumn{1}{c|}{$\#_T$} 
& \multicolumn{1}{c}{$\#_2$} 
& \multicolumn{1}{c||}{$\#_4$} 
&
$n$   
& \multicolumn{1}{c|}{$\#_T$} 
& \multicolumn{1}{c}{$\#_2$} 
& \multicolumn{1}{c}{$\#_4$} 
& \multicolumn{1}{c}{$\#_6$} 
& \multicolumn{1}{c}{$\#_8$} \\
\hline
  2& 1& 1& 0  &  20& 16& 9& 7& 0& 0 \\
  4& 1& 1& 0  &  22& 25& 16& 8& 1& 0 \\
  6& 1& 1& 0  &  24& 55& 25& 28& 1& 1 \\
  8& 2& 1& 1  &  26& 103& 55& 47& 1& 0 \\
 10& 2& 2& 0  &  28& 261& 103& 155& 3& 0 \\
 12& 3& 2& 1  &  30& 731& 261& 457& 13& 0 \\
 14& 4& 3& 1  &  32& 3295& 731& 2482& 74& 8 \\
 16& 7& 4& 3  &  34& 24147& 3295 & 19914& 938 & 0 \\
 18& 9& 7& 2  &  36& 519492 & 24147 & 436633 & 58671 & 41\\
\noalign{\hrule height0.8pt}
  \end{tabular}
}
\end{center}
\end{table}

\section{Preliminaries}
\subsection{Classification method}
\label{Subsec:CM}
Here we describe a method for classifying self-dual 
codes.
This method is similar to that given in~\cite{Huffman}.

Suppose that 
$C$ is a self-dual 
$[n,n/2,d]$ code with $d \ge 4$.
Define a subcode of $C$ as follows
\[
C_0=\{(x_1,x_2,\ldots,x_{n}) \in C \mid x_{n-1}=x_{n}\}.
\]
Since $C^\perp$ has no codeword of weight $2$,
$C_0$ has dimension $n/2-1$.
Permuting coordinates if necessary,
we may assume that there is a codeword $x=(x_1,\ldots,x_{n})$
of weight $d$ in $C$ with $x_{n-1}=x_{n}\neq0$.
Then, the following code 
\[
C_1=\{(x_1,x_2,\ldots,x_{n-2}) \mid (x_1,x_2,\ldots,x_{n}) \in C_0\}
\]
is a self-dual $[n-2,n/2-1,d-2]$ code.
Thus, the subcode $C_0$ has generator matrix of the form
\begin{equation}\label{eq:G}
G_0=
\left(\begin{array}{ccccc|cc}
&      &      & & &a_1    &a_1   \\
&      &G_1   & & &\vdots &\vdots\\
&      &      & & &a_{n/2-1} &a_{n/2-1}\\
\end{array}\right),
\end{equation}
where $G_1$ is a generator matrix of $C_1$
and $a_i \in \FF_2$ $(i=1,\ldots,n/2-1)$.
It follows that
every
self-dual $[n,n/2,d]$
code is constructed as
the code $\langle C_0, x \rangle$
for some code $C_0$ with generator
matrix of the form (\ref{eq:G})
and some vector $x \in C_0^\perp \setminus C_0$,
where $\langle C_0, x \rangle$ denotes the code
generated by the
codewords of $C_0$ and $x$.
Note that there is essentially a unique choice for
$\langle C_0, x \rangle$, for a given $C_0$.
Indeed, among the three self-dual codes lying between
$C_0^\perp$ and $C_0$, two of them are equivalent,
while the remaining code has minimum weight $2$.

In this way, 
all self-dual $[n,n/2,d]$ codes $D$,
which must be checked further for equivalence, 
are constructed, by taking generator matrices of all inequivalent 
self-dual $[n-2,n/2-1,d-2]$ codes $D_1$
as matrices $G_1$, and by considering 
$a_i \in \FF_2$ $(i=1,\ldots,n/2-1)$ in (\ref{eq:G}).

As described in~\cite{Huffman},
the number of possibilities for $a_i$ $(i=2,\ldots,n/2-1)$
is decreased 
by applying elements of
$\Aut(D_1)$ to the first $n-2$ coordinates of $D$.
This can be made more precise and more general as follows.
Two codes $C$ and $C'$ over $\FF_q$
are monomially equivalent if there is some monomial
matrix $M$ over $\FF_q$ such that $C' =C M =\{c M \mid c \in C \}$.
The monomial automorphism group of $C$
is the set of monomial matrices $M$ with $C =C M$ and
it is denoted by $\MAut(C)$.
Let $D_1$ be a linear $[n,k]$ code over $\FF_q$ with
$k\times n$ generator matrix $G_1$.
Then there exists a homomorphism $f:\MAut(D_1)\to\GL(k,q)$ defined
by $f(P)G_1=G_1P$, where $P\in\MAut(D_1)$.
The image $\image(f)$ is a subgroup of $\GL(k,q)$.
With this notation, we have the following sufficient condition
for monomial equivalence.

\begin{lem}\label{lem:aut}
Let $m$ be a positive integer, and let $a,b\in\FF_q^k$.
Suppose that $a^T$ and $b^T$ 
belong to the same $\image(f)$-orbit (under the left action),
where $a^T$ denotes the transpose of $a$.
Then the $[n+m,k]$ codes over $\FF_q$ with generator matrices
\[
\begin{pmatrix}
G_1&a^T&\cdots&a^T
\end{pmatrix}
\text{ and }
\begin{pmatrix}
G_1&b^T&\cdots&b^T
\end{pmatrix}
\]
are monomially equivalent.
\end{lem}
\begin{proof}
There exists a monomial matrix $P\in\MAut(D_1)$ such that
$a^T=f(P)b^T$. Then the monomial matrix
\[
\begin{pmatrix} P&O\\O^T&I_m\end{pmatrix}
\]
gives a monomial equivalence of the two codes above,
where $I_m$ denotes the identity matrix of order $m$ and
$O$ denotes the $n \times m$ zero matrix.
\end{proof}

In our case $(n,q)=(36,2)$, we only need to consider
$(a_1,\dots,a_{17}) \in \FF_2^{17}$ in (\ref{eq:G}),
up to the action of $\image(f)$ by Lemma~\ref{lem:aut}.
Orbit representatives for a subgroup of $\GL(17,2)$
can easily be found by {\sc Magma}~\cite{Magma}.

\subsection{Mass formula for weight enumerators}
Now we give a mass formula for weight enumerators
of self-dual codes.

\begin{lem}[Thompson~\cite{T73}]\label{lem:wt}
Let $n$ be an even positive integer.
Let $W_C(y)$ denote the weight enumerator of a code $C$.
Then
\begin{equation}\label{eq:mass}
\sum_{C}W_C(y)=
\Big(\prod_{i=1}^{n/2-1}(2^i+1)\Big)(1+y^n)+
\sum_{j=1}^{n/2-1}
\binom{n}{2j}\prod_{i=1}^{n/2-2}(2^i+1)y^{2j},
\end{equation}
where $C$ runs through
the set of all self-dual codes of length $n$.
\end{lem}

As a consequence, we have the following:

\begin{lem}\label{lem:wtC}
Let $n$ and $d$ be even positive integers. Let
$\cC$ be a family of
inequivalent self-dual codes of length
$n$ and minimum weight at most $d$. Then $\cC$ is
a complete set of representatives for equivalence classes
of self-dual codes of length
$n$ and minimum weight at most $d$, if and only if
\begin{equation}\label{eq:wtd}
\sum_{C\in\cC}\frac{n!}{\#\Aut(C)}
\#\{x\in C\mid \wt(x)=d\}=
\binom{n}{d}
\prod_{i=1}^{n/2-2}(2^i+1).
\end{equation}
\end{lem}
\begin{proof}
Consider the coefficient of $y^d$ in the formula
(\ref{eq:mass}) in Lemma~\ref{lem:wt}.
\end{proof}

\section{Classification of self-dual codes of length 36}
\label{Sec:36}

In this section, we give a complete classification of 
self-dual codes of length $36$.

Any self-dual code of length $n+2$ and minimum weight $2$
is decomposable as $i_2\oplus C_{n}$,
where $i_2$ is the unique self-dual code of length $2$
and $C_{n}$ is some self-dual code of length $n$.
Since there are $24147$ inequivalent self-dual codes of 
length $34$~\cite{B06},
there are $24147$ inequivalent self-dual $[36,18,2]$ codes.
We denote the set of these $24147$ codes by $\mathcal{C}_{36,2}$.

For each self-dual $[34,17,2]$ code given in~\cite{B06},
the method given in Subsection~\ref{Subsec:CM}
produces a number of self-dual $[36,18,4]$ codes.
We continue the process until we obtain a set $\mathcal{C}_{36,4}$ of 
inequivalent self-dual $[36,18,4]$ codes such that 
$\mathcal{C}=\mathcal{C}_{36,2} \cup\mathcal{C}_{36,4}$
satisfies (\ref{eq:wtd}).
Lemma~\ref{lem:wtC} implies that
there is no other self-dual $[36,18,4]$ code.

Similarly, we found the set $\mathcal{C}_{36,6}$
of the $58671$ inequivalent self-dual $[36,18,6]$ codes 
from the set of inequivalent self-dual $[34,17,4]$ codes.
Setting $\mathcal{C}=\mathcal{C}_{36,2} \cup\mathcal{C}_{36,4}
\cup\mathcal{C}_{36,6}$ in Lemma~\ref{lem:wtC}, 
one can verify that
there is no other self-dual $[36,18,6]$ code.


{From} our results, together with the set of extremal self-dual
codes found by~\cite{MG08},
we obtain the set $\mathcal{C}_{36}$ of $519492$
inequivalent self-dual codes satisfying
\[
\sum_{C \in \mathcal{C}_{36}}
\frac{36!}{\#\Aut(C)} =
\prod_{i=1}^{17}(2^i+1),
\]
%
which is the usual mass formula appearing as the
constant term of (\ref{eq:mass}).
Since this constant term
gives the number of distinct self-dual
codes of length $36$,
it follows that there is no other self-dual code
of length $36$.
Therefore, we have Theorem~\ref{thm:36}.

\section{Some properties}

The weight enumerator of a self-dual code of length $36$
can be written as
\begin{align*}
&
1
+ \alpha y^2
+ (12 \alpha + \beta )y^4
+ (64 \alpha  + 6 \beta  + \gamma)y^6
+ (33  +  196 \alpha + 11 \beta + 64 \delta)y^8
\\ &
+ (3168 + 364 \alpha - 4 \beta - 6 \gamma - 384 \delta)y^{10}
+ (7059 + 364 \alpha - 39 \beta + 832 \delta)y^{12}
\\ &
+ (30336 - 38 \beta + 15 \gamma - 512 \delta)y^{14}
+ (58443 - 572 \alpha+ 27 \beta - 896 \delta)y^{16}
\\ &
+ (64064 - 858 \alpha+ 72 \beta - 20 \gamma + 1792 \delta)y^{18}
+ \cdots + y^{36},
\end{align*}
where $\alpha,\beta,\gamma,\delta$ are integers.
The numbers of distinct weight enumerators of
self-dual codes of length $36$ are listed 
in Table~\ref{Tab:WE} for each minimum weight $d$.
In particular, we list in Table~\ref{Tab:WEd6}
the numbers of self-dual codes with $d=6$
for each weight enumerator, where
the numbers $\#$ of codes and $(\gamma,\delta)$ are listed.

\begin{table}[th]
\caption{Numbers of weight enumerators}
\label{Tab:WE}
\begin{center}
{\small
\begin{tabular}{c|cccc}
\noalign{\hrule height0.8pt}
$d$  & $2$ & $4$ & $6$ & $8$ \\ 
\hline
\#   & 1264 & 2210 &  28 &    2 \\
\noalign{\hrule height0.8pt}
  \end{tabular}
}
\end{center}
\end{table}

\begin{table}[th]
\caption{Numbers of weight enumerators for $d=6$}
\label{Tab:WEd6}
\begin{center}
{\small
\begin{tabular}{rc|rc|rc|rc}
\noalign{\hrule height0.8pt}
\multicolumn{1}{c}{$\#$}  & $(\gamma,\delta)$ &
\multicolumn{1}{c}{$\#$}  & $(\gamma,\delta)$ &
\multicolumn{1}{c}{$\#$}  & $(\gamma,\delta)$ &
\multicolumn{1}{c}{$\#$}  & $(\gamma,\delta)$ \\
\hline
 107&$(2,3)$&   257&$( 8,4)$&  8493&$(16,3)$& 146&$(28,3)$\\
  41&$(2,4)$&  7710&$(10,3)$&     1&$(16,4)$& 122&$(30,3)$\\
 559&$(4,3)$&   183&$(10,4)$&  6432&$(18,3)$&  20&$(32,3)$\\
 111&$(4,4)$&  9739&$(12,3)$&  3773&$(20,3)$&  25&$(34,3)$\\
1971&$(6,3)$&    82&$(12,4)$&  2319&$(22,3)$&   4&$(36,3)$\\
 214&$(6,4)$& 10262&$(14,3)$&   954&$(24,3)$&   5&$(38,3)$\\
4535&$(8,3)$&    22&$(14,4)$&   579&$(26,3)$&   5&$(42,3)$\\
\noalign{\hrule height0.8pt}
  \end{tabular}
}
\end{center}
\end{table}

The smallest order $\#\Aut_s$ and the largest order $\#\Aut_l$
among automorphism groups of self-dual codes of length $36$ 
are listed in Table~\ref{Tab:Aut} for each minimum weight $d$.
%
In particular, for $d=6$, 
the number $N$ of the codes with an automorphism group of 
order $\#\Aut$ is listed in Table~\ref{Tab:Autd6}.
There is no self-dual code with a trivial automorphism group
for lengths up to $32$ (see~\cite{BR02}).
At length $34$, there are $159$ inequivalent self-dual 
$[34,17,6]$ codes with trivial automorphism groups.
Compared to self-dual codes of length $34$,
there are a great number of self-dual 
codes with trivial automorphism groups for length $36$.

\begin{table}[th]
\caption{Orders of the automorphism groups}
\label{Tab:Aut}
\begin{center}
{\small
\begin{tabular}{c|cccc}
\noalign{\hrule height0.8pt}
$d$        & $2$ & $4$ & $6$ & $8$ \\
\hline
$\#\Aut_s$   & $2$ & $4$ & $1$ & $6$ \\
$\#\Aut_l$   &
$2^{18}\cdot18!$& $2^{17}\cdot18!$ & 21504  & 34560  \\
\noalign{\hrule height0.8pt}
  \end{tabular}
}
\end{center}
\end{table}

\begin{table}[th]
\caption{Orders of the automorphism groups for $d=6$}
\label{Tab:Autd6}
\begin{center}
{\footnotesize
\begin{tabular}{cc|cc|cc|cc|cc|cc}
\noalign{\hrule height0.8pt}
$\#\Aut$ & $N$ &
$\#\Aut$ & $N$ &
$\#\Aut$ & $N$ &
$\#\Aut$ & $N$ &
$\#\Aut$ & $N$ &
$\#\Aut$ & $N$ \\
\hline
  1& 41019& 14&   1&  56&   1& 192& 25&  576&  2& 3456&  1\\
  2& 11242& 16& 643&  64& 118& 240&  3&  768& 12& 4608&  1\\
  3&    37& 18&   3&  72&   7& 256& 21&  864&  1& 5376&  1\\
  4&  3368& 20&   2&  80&   1& 288&  7& 1152&  4& 5760&  1\\
  6&   137& 24&  59&  96&  43& 336&  1& 1344&  1&12960&  2\\
  7&     2& 32& 251& 108&   1& 384& 18& 1536&  5&21504&  1\\
  8&  1297& 36&  21& 128&  45& 432&  1& 1728&  3&\\
 12&   166& 48&  78& 144&   9& 512&  8& 2304&  1&\\
\noalign{\hrule height0.8pt}
  \end{tabular}
}
\end{center}
\end{table}

Let $C$ be a singly even self-dual code and
let $C_0$ denote the 
subcode of codewords having weight $\equiv0\pmod4$.
Then $C_0$ is a subcode of codimension $1$.
The {\em shadow} $S$ of $C$ is defined to be 
$C_0^\perp \setminus C$.
Let $d$ and $s$ denote the minimum weights of a
self-dual code of length $36$
and its shadow, respectively.
It was shown in~\cite{BG04} that
$2d+s \le 22$.
The numbers $\#_s$ of self-dual codes with shadows
of minimum weight $s$ are listed in Table~\ref{Tab:S}
for each minimum weight $d$.
Note that there is no self-dual $[36,18,4]$ code
meeting the bound.
A classification of self-dual $[36,18,6]$ codes
meeting the bound can be found in~\cite{MG08}.

\begin{table}[th]
\caption{Minimum weights of the shadows}
\label{Tab:S}
\begin{center}
{\small
\begin{tabular}{c|rrrrr}
\noalign{\hrule height0.8pt}
$d$   
& \multicolumn{1}{c}{$\#_2$} 
& \multicolumn{1}{c}{$\#_6$} 
& \multicolumn{1}{c}{$\#_{10}$} 
& \multicolumn{1}{c}{$\#_{14}$} 
& \multicolumn{1}{c}{$\#_{18}$} \\
\hline
2 &   679 &  22883 & 577 & 7 & 1 \\
4 & 22541 & 414068 &  24 & 0 & - \\
6 &   911 &  57755 &   5 & - & - \\
8 &    16 &     25 &   - & - & - \\

\noalign{\hrule height0.8pt}
  \end{tabular}
}
\end{center}
\end{table}

The covering radius $R(C)$ of a code $C$ is the smallest integer $R$
such that spheres of radius $R$ around codewords of $C$ cover the
space $\FF_2^n$.
The covering radius is a basic and important geometric parameter 
of a code (see~\cite{CKMS}).
Let $C$ be a self-dual $[36,18,d]$ code.
By~\cite[Eq.~(2)]{CKMS} and
the Delsarte bound (see~\cite[Theorem 2]{CKMS}),
\[
6 \le R(C) \le 20-d.
\]
The numbers $\#R_r$ of self-dual codes of length $36$ 
with covering radii $r$
are listed in Table~\ref{Tab:CR} for each minimum weight $d$.
There is a unique self-dual $[36,18,6]$ code with
covering radius $6$.
This code $C_{36}$ has generator matrix $(\ I_{18}\ ,\  M\ )$
where $M$ is listed in Figure~\ref{Fig}.
The code $C_{36}$ has weight enumerator 
with $(\alpha, \beta, \gamma, \delta)=(0, 0, 12, 4)$, 
it has shadow of minimum weight $2$
and it has automorphism group of order $5760$.

\begin{table}[th]
\caption{Covering radii of self-dual codes of length $36$}
\label{Tab:CR}
\begin{center}
{\small
\begin{tabular}{c|rrrrrrr}
\noalign{\hrule height0.8pt}
$d$   & 
\multicolumn{1}{c}{$\#R_{6}$} & 
\multicolumn{1}{c}{$\#R_{7}$} & 
\multicolumn{1}{c}{$\#R_{8}$} & 
\multicolumn{1}{c}{$\#R_{9}$} & 
\multicolumn{1}{c}{$\#R_{10}$} &
\multicolumn{1}{c}{$\#R_{11}$} & 
\multicolumn{1}{c}{$\#R_{12}$} \\
\hline
2& 0 &     23& 20148& 3010& 830& 87& 34 \\
4& 23& 372396& 63599&  587&  28&  0&  0 \\
6& 1 &  53226&  5439&    0&   5&  0&  0 \\
8& 3 &     38&     0&    0&   0&  0&  0 \\
\hline
$d$   & 
\multicolumn{1}{c}{$\#R_{13}$} & 
\multicolumn{1}{c}{$\#R_{14}$} &
\multicolumn{1}{c}{$\#R_{15}$} & 
\multicolumn{1}{c}{$\#R_{16}$} & 
\multicolumn{1}{c}{$\#R_{17}$} & 
\multicolumn{1}{c}{$\#R_{18}$} \\
\hline
2&  5& 7& 1& 1& 0& 1 \\
4&  0& 0& 0& 0& -& - \\
6&  0& 0& -& -& -& - \\
\noalign{\hrule height0.8pt}
  \end{tabular}
}
\end{center}
\end{table}

\begin{figure}[htb]
\centering
{\small
\[
M=
\left(\begin{array}{c}
 0 0 1 1 0 0 0 0 0 0 1 0 1 0 0 0 1 0\\
 0 0 1 1 0 0 0 0 0 0 1 0 1 0 1 1 0 1\\
 0 0 0 1 1 0 0 0 0 0 0 1 1 1 0 1 0 1\\
 0 0 0 1 1 0 0 0 0 0 0 1 0 0 0 1 1 0\\
 0 0 1 0 0 0 0 0 0 0 0 1 1 0 0 1 0 1\\
 0 0 1 0 0 0 0 0 0 0 1 0 0 1 1 0 0 1\\
 1 0 1 1 1 0 0 0 1 0 0 0 1 1 1 1 1 1\\
 1 0 1 1 1 0 1 1 0 1 1 1 1 1 1 1 0 0\\
 1 1 0 0 0 1 0 0 0 1 0 0 1 0 0 1 1 0\\
 1 1 0 0 1 0 1 1 1 0 1 1 0 1 0 1 1 0\\
 0 1 0 0 1 1 0 1 1 0 0 0 1 1 0 0 1 1\\
 0 1 1 1 0 0 0 1 0 1 1 1 0 0 1 1 1 1\\
 0 0 1 0 1 1 1 0 0 1 1 1 0 0 0 0 0 0\\
 1 1 1 0 1 1 1 0 1 0 0 0 0 0 0 0 1 1\\
 0 1 0 0 0 0 1 0 1 1 1 1 1 0 0 1 0 1\\
 1 0 1 1 1 1 0 1 1 1 1 1 0 1 1 0 1 0\\
 1 1 0 1 1 1 1 1 0 1 1 1 1 0 1 0 1 0\\
 0 0 1 0 1 1 1 1 1 0 1 1 1 0 0 1 1 0
\end{array}\right)
\]
\caption{A self-dual $[36,18,6]$ code with covering radius $6$}
\label{Fig}
}
\end{figure}

We end this paper with some remark on the classification of
self-dual codes of length $38$.
Since 
\[
\frac{\prod_{i=1}^{18}(2^i+1)}{38!}>13644432.203,
\]
there are at least $13644433$
inequivalent self-dual codes of
length $38$.

\bigskip
\noindent
{\bf Acknowledgment.} 
This work of the first author was supported by JST PRESTO program.



\end{document}